\documentclass[12pt,reqno]{amsart}
\usepackage{amssymb}
\usepackage[curve,matrix,arrow]{xy}

\baselineskip

\theoremstyle{plain} 
\numberwithin{equation}{section}
\newtheorem{thm}[equation]{Theorem}
\newtheorem{lemma}[equation]{Lemma}
\newtheorem{cor}[equation]{Corollary}
\newtheorem{prop}[equation]{Proposition}

\newtheorem{defi}[equation]{Definition}
\newtheorem{hyp}[equation]{Hypothesis}

\theoremstyle{definition}
\newtheorem{rem}[equation]{Remark}

\theoremstyle{remark}

\def\CC{{\mathcal{C}}}

\def\CE{{\mathcal{E}}}
\def\CF{{\mathcal{F}}}

\def\CI{{\mathcal{I}}}

\def\CM{{\mathcal{M}}}

\def\CS{{\mathcal{S}}}

\def\CV{{\mathcal{V}}}

\def\CV{{\mathcal{V}}}

\def\bF{{\mathbb F}}

\def\bP{{\mathbb P}}
\def\bZ{{\mathbb Z}}

\def\Proj{\operatorname{Proj}\nolimits}

\def\res{\operatorname{res}\nolimits}

\def\Rad{\operatorname{Rad}\nolimits}
\def\Soc{\operatorname{Soc}\nolimits}

\def\modg{\operatorname{{\bf mod}(\text{$kG$})}\nolimits}

\def\Modg{\operatorname{{\bf Mod}(\text{$kG$})}\nolimits}
\def\modh{\operatorname{{\bf mod}(\text{$kH$})}\nolimits}
\def\stmodg{\operatorname{{\bf stmod}(\text{$kG$})}\nolimits}
\def\Stmodg{\operatorname{{\bf StMod}(\text{$kG$})}\nolimits}

\def\stmod{\operatorname{{\bf stmod}}\nolimits}

\def\HHH{\operatorname{H}\nolimits}
\def\Hom{\operatorname{Hom}\nolimits}
\def\PHom{\operatorname{PHom}\nolimits}

\def\End{\operatorname{End}\nolimits}
\def\HH#1#2#3{\HHH^{#1}(#2,#3)}
\def\HHHH{\operatorname{\widehat{\HHH}}\nolimits}

\def\Homul{\operatorname{\underline{Hom}}\nolimits}

\def\Endul{\operatorname{\underline{End}}\nolimits}

\def\Ext{\operatorname{Ext}\nolimits}

\def\hgs{\HH{*}{G}{k}}
\def\hes{\HH{*}{E}{k}}

\begin{document}

\title[Endomorphism ring of the trivial module]
{The endomorphism ring of the trivial module in a localized category}
\author[Jon F. Carlson]{Jon F. Carlson}
\address{Department of Mathematics, University of Georgia, 
Athens, GA 30602, USA}
\email{jfc@math.uga.edu}
\thanks{Research partially supported by 
Simons Foundation grant 054813-01}

%%\date\today

\begin{abstract}
Suppose that $G$ is a finite group and $k$ is a field of characteristic
$p >0$. Let $\CM$ be the thick tensor ideal of finitely generated 
modules whose support variety is in a fixed subvariety $V$ of the projectivized
prime ideal spectrum $\Proj \HHH^*(G,k)$. Let $\CC$ denote the Verdier 
localization of the stable module category $\stmod(kG)$ at $\CM$. We 
show that if $V$ is a finite collection of closed points and if 
the $p$-rank every  maximal elementary abelian $p$-subgroups of $G$
is at least 3, then the endomorphism ring of the trivial 
module in $\CC$ is a local ring whose unique maximal ideal is infinitely
generated and nilpotent. In addition, we show an example where the 
endomorphism ring in $\CC$ of a compact object is not finitely presented as a 
module over the endomorphism ring of the trivial module.
\end{abstract}
\maketitle

\section{Introduction}
Suppose that $G$ is a finite group and that $k$ is a field of characteristic
$p > 0$. The stable category $\stmodg$ of finitely generated $kG$-modules 
is a tensor triangulated category. A thick tensor ideal in $\stmodg$
is determined by the support variety of its objects. Hence, for any 
closed subvariety $V$ in $V_G(k) = \Proj \hgs$, the full subcategory 
$\CM_V$ of all 
finitely generated $kG$-modules $M$ with $V_G(M) \subset V$, is a 
thick subcategory that is closed under tensor product with any 
finitely generated $kG$-module. Moreover, every thick tensor ideal can
be defined in this or a similar way. 
Associated to $\CM_V$ is a distinguished triangle 
\[
\xymatrix{ 
{} \ar[r] & \CE_V \ar[r] & k \ar[r] & \CF_V \ar[r] & {}
}
\]
where $\CE_V$ and $\CF_V$ are idempotent $kG$-modules that are almost 
always infinitely generated. 
In addition, $\CE_V \otimes M \cong M$ in the stable category if and 
only if $V_G(M) \subseteq V$. Tensoring with 
$\CF_V$ is the localizing functor 
to the Verdier localization $\CC_V$ of $\stmodg$ at $\CM_V$. Thus the 
localized category $\CC_V$ is embedded in the stable category 
$\Stmodg$ of all $kG$-modules. 

In the localized category $\CC_V$, the trivial 
module $k$ is identified with $\CF_V$ and the ring $\End_{\CC_V}(k)$ is 
isomorphic to $\End_{\Stmodg}(\CF_V)$. For $kG$-modules $M$ and $N$, 
the group $\Hom_{\CC_V}(M,N)$ is a module over $\End_{\CC_V}(k)$.
This suggests that modules in the category $\CC_V$ can be distinguished
by invariants such as the annihilators of their endomorphism rings or 
cohomology rings. Such is the essence of the support variety theory that
classifies the thick tensor ideals in $\stmod(kG)$. In subsequent work
\cite{Clocvar}, we show that there is such a theory that is nontrivial in 
the case of the colocalized category generated by $\CE_V$. However, 
in the localized category, there are additional complications, as we 
see in this paper.   

In this paper, we complete the task of characterizing $\End_{\CC_V}(k)$
in the case that the maximal elementary abelian subgroups of  
$G$ have sufficiently large $p$-rank and 
the variety $V$ is a finite collection of closed points. 
We prove in such cases that the 
ring $\End_{\CC_V}(k)$ is a local ring whose maximal ideal is 
infinitely generated and nilpotent. 

Our study relies on earlier results in \cite{Ctriv} that proves the 
special case in which $V$ is a single closed point in $V_G(k)$ and 
$G$ is elementary abelian. More generally, that paper shows that 
the nonpositive Tate cohomology ring of any finite group $H$ can be 
realized as the endomorpism ring of the trivial  
module in the Verdier localization $\CC_V$ where $V$ is a single point
in the spectrum of the cohomology ring of $G = C \times H$ for $C$
a cyclic group of order $p$. The proof of our main theorem also 
requires the fact that nilpotence
in cohomology can be detected on restrictions to 
elementary abelian $p$-subgroups of a finite group \cite{CQ}. This 
theorem does not hold for $G$ a general finite group scheme, and hence 
the proof of our main theorem does not extend to that realm.

The next section presents an introduction and references to the 
categories and recalls some theorems on support varieties.
In the three sections that follow, we review the main theorem of \cite{Ctriv} 
and extend the result to the case in which the variety $V$ is a 
finite collection of more than one closed points. In section 5,
we prove the main theorem for any finite group whose maximal 
elementary abelian $p$-subgroups have $p$-rank at least three. 
In section 6, we look at the restriction of the module $\CF_V$ from
an elementary abelian group to one of its proper subgroup. This result 
is used in the final section to show by example that even for a 
compact object $M$ in $\stmodg$, it is possible that $\End_{\CC_V}(M)$
is not finitely generated over $\End_{\CC_V}(k)$. In the other direction,
we show also in Section 7 that if $V$ is the subvariety of all 
homogeneous prime ideals that contain a single non-nilpotent element
of cohomology, then for any compact object $M$, $\End_{\CC_V}(M)$
is finitely generated over $\End_{\CC_V}(k)$. In the final section 
we present an example that shows more of the stucture of the 
idempotent module $\CF_V$ in the case of groups that are not elementary
abelian $p$-groups. 

We would like to thank Paul Balmer for helpful conversations
and information. 

\section{Background}
In this section, we review some background. As references, we refer the
reader to \cite{CTVZ} or \cite{Bbook} 
for information on the cohomology of finite groups
and support varieties in this context. For information on triangulated
categories see \cite{Neem}. A lot of the background material is
summarized very well in the paper \cite{BF} of Balmer and Favi.

Throughout the paper, we let $G$ be a finite group and $k$ a field of
characteristic $p > 0$. For convenience, we assume that $k$ is algebraically
closed. Recall that $kG$ is a Hopf algebra so that if $M$ and $N$ are
$kG$-modules, then so is $M \otimes_k N$. In general, we write
$\otimes$ for $\otimes_k$.

Let $\modg$ denote the category of finitely generated $kG$-modules
and $\Modg$ the category of all $kG$-module. Let $\stmodg$ be the stable
category of finitely generated $kG$-module modulo projectives. The objects in
$\stmodg$ are the same as those in $\modg$. If $M$ and $N$ are finitely
generated $kG$-modules, then the group of morphisms 
from $M$ to $N$ in the stable category is
the quotient $\Homul_{kG}(M,N) = \Hom_{kG}(M,N)/\PHom_{kG}(M,N)$ where
$\PHom_{kG}(M,N)$ is the set of homomorphisms that factor through projective
modules. The definition of the stable  category of all modules
$\Stmodg$ is similar.

The stable categories $\stmodg$ and $\Stmodg$ are tensor triangulated
categories. The tensor is the one given by the Hopf algebra structure on
$kG$ as mentioned above. Triangles correspond roughly to exact sequences
in the module categories. The translation functor for both is $\Omega^{-1}$,
so that a triangle looks like
\[
\xymatrix{
A \ar[r] & B \ar[r] & C \ar[r] & \Omega^{-1}(A)
} \]
where for some projective module $P$ there is an exact sequence
$0 \to A \to B \oplus P \to C \to 0$. Here, $\Omega^{-1}(M)$
is the cokernel of an injective hull $M \hookrightarrow I$ for $I$ Injective. 

The cohomology ring $\HHH^*(G,k)$ is a finitely generated,
graded-commutative algebra over $k$. Let $V_G(k) = \Proj(\HHH^*(G,k)$
be its projectivized prime ideal spectrum, the collection of all
homogeneous prime ideals with the Zariski topology. The support
variety of a finitely generated $kG$-module $M$ is the closed subvariety
consisting of all homogeneous prime ideals that contain the
annihilator of $\Ext^*_{kG}(M,M)$ in $\HHH^*(G,k)$. The support variety
of an infinitely generated $kG$-module is a subset of $V_G(k)$,
not necessarily closed (see \cite{BCR2}).

If $H$ is a subgroup of $G$, the restriction functor $\modg \to \modh$
induces a map on cohomology ring $\res_{G,H}:\HHH^*(G,k) \to 
\HHH^*(H,k)$ and also a map on sectra $\res^*_{G,H}:V_H(k) 
\to V_G(k)$. 

For much of the next few sections, we assume
that $G = \langle g_1, \dots, g_r \rangle$ is an
elementary abelian $p$-group of order $p^r$. 
In this case, we set $X_i = g_i -1 \in kG$ for $i = 1, \dots, r.$
Then $X_i^p = 0$ and $kG \cong k[X_1, \dots, X_r]/(X_1^p, \dots, X_r^p)$.
With this structure in mind, we make the following definition.

\begin{defi} \label{def:flatalg}
Suppose that $kG$ is the group algebra of an elementary abelian $p$-group. 
A flat subalgebra of $kG$ is the image in $kG$ of a flat map 
$\alpha: k[t_1, \dots, t_s]/(t_1^p, \dots, t_s^p) \to kG$. We say 
a flat subalgebra is maximal if $s= r-1$ where $r$ is the $p$-rank 
of $G$.  
\end{defi}

By definition, a map $\alpha$, as above, is flat if $kG$ is a 
projective module over the image of the ring 
$k[t_1, \dots, t_s]/(t_1^p, \dots, t_s^p)$. 
This happens if and only if the images $\alpha(t_1), \dots, \alpha(t_s)$,
in $\Rad(kG)/\Rad^2(kG)$ are $k$-linearly independent.   In particular,
we have the following. 

\begin{lemma} \label{lem:complement}
Suppose that $G$ is an elementary abelian $p$-group of $p$-rank $r$.
Let $\alpha: k[t_1, \dots, t_s]/(t_1^p, \dots, t_s^p) \to kG$ be a 
flat map. Then there exists another flat map 
$\beta: k[t_1, \dots, t_{r-s}]/(t_1^p, \dots, t_{r-s}^p) \to kG$
such that $kG$ is the internal tensor product $kG = A \otimes B$, 
where $A$ and $B$ are the images of $\alpha$ and $\beta$, respectively. 
\end{lemma}

\begin{proof}
Choose elements $m_1, \dots, m_{r-s}$ in $kG$, such that the  classes
modulo $\Rad^2(kG)$ of  
$\alpha(t_1), \dots, \alpha(t_s), m_1, \dots, m_{r-s}$
form a basis for $\Rad(kG)/\Rad^2(kG)$. Then let $\beta$ be defined
by $\beta(t_i) = m_i$ for $i = 1, \dots, s$. Then $AB = kG$, by 
Nakayama's Lemma and a dimension argument. 
\end{proof}

If $\alpha$, as above, is a flat map, then 
the multiplicative subgroup generated
by the images $\alpha(1+t_i)$ is called a shifted subgroup of $kG$ in
other papers. It is an elementary abelian $p$-subgroup of the group
of units of $kG$.  In the case that $s = 1$, we have an 
example of a $\pi$-point. 

\begin{defi} \label{defi:pipoint} \cite{FP}
A $\pi$-point is a flat map $\alpha_K: K[t]/(t^p) \to KG_K$
where $K$ is an extension of the field $k$. If $G$ 
is a finite group scheme that is not
elementary abelian, then we assume also that $\alpha_K$ factors
by flat maps through a unipotent abelian subgroup scheme of $KG_K$.
Two $\pi$-points $\alpha_K: K[t]/(t^p) \to KG_K$ and 
$\beta_L: L[t]/(t^p) \to LG_L$ are equivalent if for any finitely generated
$kG$-module $M$, the restriction $\alpha_K^*(K \otimes M)$ is 
projective if and only if $\beta^*_L(L \otimes M)$ is projective.
\end{defi}

The set of equivalence classes of $\pi$-point has a partial order 
coming from specializations, and that ordering gives the set a 
topology. With this in mind we have the following, which holds for 
any finite group scheme $G$. 

\begin{thm} \label{thm:rankvar} \cite{FP}
The space of equivalence classes of $\pi$-points is homeomorphic 
to $V_G(k) = \Proj \hgs.$
\end{thm}

The point is that if $A = K[t]/(t^p)$, then $\HHH^*(A,K)/\Rad(\HHH^*(A,K))$
is a polynomial ring in one variable. So, if $\alpha: A \to KG$ is a 
$\pi$-point, then the kernel of the composition
\[
\xymatrix{
\hgs \ar[r]^{\alpha^*} & \HHH^*(A,K) \ar[r] & \HHH^*(A,K)/\Rad(\HHH^*(A,K))
}
\] 
is a prime ideal. Equivalent $\pi$-points determine the same prime ideal.

With the identification given by the theorem, we can define the support
variety $\CV_G(M)$ of any $kG$-module $M$ 
to be the set of all equivalence classes 
of $\pi$-point $\alpha_K:K[t]/(t^p) \to KG_K$ such that the restriction
$\alpha_K^*(K \otimes M)$ is not a free $KG$-module. In the case that 
$M$ is finitely generated, $\CV_G(M) \simeq V_G(M)$ is a closed set. 

\begin{rem}\label{rem:quillen}
If $G$ is a finite group that is not elementary abelian, then
the Quillen Dimension Theorem (see \cite{Quil} or 
\cite[Theorem 8.4.6]{CTVZ}) says
that $V_G(k) = \Proj \HHH^*(G,k) = \cup \res^*_{G,E}(V_E(k))$,
where the union is over the elementary abelian $p$-subgroups $E$ of $G$. 
This assures us that every $\pi$-point is equivalent to one that 
factors through the inclusion of the group algebra of some 
elementary abelian $p$-subgroup $E$ of $G$ into $kG$. Or, stated another
way, every homogeneous prime ideal in $\hgs$ contains the kernel
of the restriction $\res_{G,E}:\hgs \to \hes,$ for some elementary 
abelian $p$-subgroup $E.$
\end{rem}

\section{Point varieties}
A subcategory $\CM$ of a triangulated category $\CC$
is thick if it is triangulated
and closed under taking direct summands. It is a thick tensor ideal if it
is thick and if, for any $X \in \CC$ and $Y \in \CM$, $X \otimes Y$ is in
$\CM$. For $V$ a
closed subset of $V_G(k)$, let $\CM_V$ be the thick tensor ideal
in $\stmodg$ consisting of all finitely generated $kG$-modules $M$
with $V_G(M) \subseteq V$. More generally, let $\CV$ be a collection
of closed subsets $V_G(k)$ that is closed under taking finite unions
and specializations (meaning that if $U \subseteq V \in \CV$ then
$U \in \CV$). Then the subcategory $\CM_{\CV}$ of all finitely generated
modules $M$ with $V_G(M) \in \CV$ is a thick tensor ideal. Indeed, this is 
the story. 

\begin{thm} \label{thm:bcr} \cite{BCR}
If $\CM$ is a thick tensor ideal in $\stmodg$, then $\CM = \CM_{\CV}$
for some collection $\CV$ of closed subsets of $V_G(k)$ that is closed
under finite unions and specializations.
\end{thm}

Corresponding to a thick tensor ideal $\CM_\CV$ in $\stmodg$ is a triangle of
idempotent modules in $\Stmodg$ having the form
\[
\xymatrix{
\CS_\CV:& {} \ar[r]  & \CE_\CV \ar[r]^{\sigma_\CV} & k \ar[r]^{\tau_\CV} &
\CF_\CV \ar[r] & {}.
}
\]
See \cite{R} for proofs and details. 
The modules $\CE_\CV$ and $\CF_\CV$ are idempotent in the stable category,
meaning that $\CE_\CV \otimes \CE_\CV \cong \CE_\CV$ and
$\CF_\CV \otimes \CF_\CV \cong \CF_\CV$ in $\Stmodg$, {\it i. e.}
ignoring projective summands. In addition, $\CE_\CV \otimes \CF_\CV \cong 0$
in the stable category. The support variety $\CV_G(\CE_{\CV})$ is the set
of all equivalences classes of $\pi$-points 
corresponding to irreducible closed subsets in $\CV$,
and $\CV_G(\CF_{\CV}) = V_G(k) \setminus \CV_G(\CE_{\CV})$. 

For any finitely generated $kG$ module $X$,
the triangle
\[
\xymatrix{
X \otimes \CS_\CV: & \CE_\CV(X) \ar[r]^{\quad \mu_X} & X
\ar[r]^{\nu_X \quad} & \CF_\CV(X) \ar[r] & {}
}
\]
has a couple of universal properties \cite{R}.
Let $\CM^{\oplus}$ denote the closure of $\CM$ in $\Stmodg$ under
taking arbitrary direct sums. The map $\mu_X$ is universal for maps
from objects in $\CM_\CV^\oplus$ to $X$, meaning that if $Y$ is in
$\CM^{\oplus}_\CV$, then any map $Y \to X$ factors through $\mu_X$. The map
$\nu_X$ is universal for maps from $X$ to $\CM_\CV$-local objects, meaning
objects $Y$ such that $\Homul_{kG}(M, Y) = \{0 \}$ for all $M$ in
$\CM_\CV$. The universal property says that for an $\CM_{\CV}$-local module $Y$,
any map $X \to Y$ factors through $\nu_X$.

In the event that $V$ is a closed subset of $V_G(k)$, let $\CE_V = \CE_\CV$
and $\CF_V = \CF_{\CV}$ where $\CV$ is the collection of all closed subsets
of $V$.

\begin{lemma}\label{lem:Elocal}
Suppose that $V$ is a closed subvariety of $V_G(k)$. Suppose that $L$ is
a $kG$-module such that $U \cap V = \emptyset$ for all $U \in \CV_G(L)$. Then
$\Homul_{kG}(\CE_V, L) = \{0\}$.
\end{lemma}

\begin{proof}
The point is that $\CE_V$ can be constructed as the direct limit of
finitely generated modules having variety equal to $V$. So $L$ is
$\CM_V$-local.
\end{proof}

Suppose that $\CM = \CM_{\CV}$ is a thick tensor ideal of $\stmodg$ for an
appropriate collection $\CV$. The Verdier localization
$\CC = \CC_{\CV}$ of $\stmodg$
with respect to $\CM$ is the category whose objects are the same as
those of $\stmodg$. The collection of
morphisms from an object $M$ to an object $N$ is obtained by inverting
any morphism with the property that the third object in the triangle of
that morphism is in $\CM$. Thus, objects in $\CM$ are equal to the zero
object in $\CC$. One of the motivations for this work is that
$\Endul(\CF_V)$ is isomorphic to the ring of endomorpisms of the
trivial module $k$ in the localized category $\CC$.

\begin{prop} \label{prop:nointersect}
Suppose that $V = V_1 \cup V_2$ where $V_1$ and $V_2$ are closed
subvarieties such that $V_1 \cap V_2 = \emptyset$. Then $\CF_V$
is the pushout of the diagram
\[
\xymatrix{
k \ar[r]^{\tau_{V_1}} \ar[d]^{\tau_{V_2}} & \CF_{V_1} \ar[d] \\
\CF_{V_2} \ar[r] & \CF_{V}
}
\]
That is, $\CF_V \cong (\CF_{V_1} \oplus \CF_{V_2})/ N$ where
$N = \{ (\tau_{V_1}(a), -\tau_{V_2}(a)) \ \vert \ a \in k \}.$
\end{prop}

\begin{proof}
The thing to note is that $\CE_{V} \cong \CE_{V_1} \oplus \CE_{V_2}.$
That is, because, $V_1 \cap V_2 = \emptyset$, 
if $M \in \CM_V$, then $M \cong M_1 \oplus M_2$ where
$M_i \in \CM_{V_i}$ for $i = 1,2$. So in particular, the map
$\CE_{V_1} \oplus \CE_{V_2} \to k$ sending $(u, v)$ to
$\sigma_{V_1}(u) + \sigma_{V_2}(v)$ has the desired universal
property. The third object in the triangle of the map is the pushout,
and it also satisfies the desired universal property. Moreover,
we know that $\CE_{V_1} \otimes \CE_{V_2}$ is projective because the
varieties of the two modules are disjoint \cite{BCR2}.
So $\CE_{V_1} \oplus \CE_{V_2}$
is an idempotent module. This is sufficient to prove the proposition.
\end{proof}

\begin{lemma} \label{lem:restr-idem} 
Let $G$ be an elementary abelian group of order $p^r$.
Suppose that $H$ is a subgroup of $G$ or that $kH$ is the image of a 
flat map $\gamma:k[t_1, \dots, t_s]/(t_1^p, \dots, t_s^p) \to kG$.
Let $V$ be a closed subvariety of $V_G(k)$.
Then the restriction of the exact triangle $\CS_V$ to $kH$ is the triangle
\[
\xymatrix{
\CS_{V^\prime}:& {} \ar[r]  & \CE_{V^\prime} \ar[r]^{\sigma_{V^\prime}} 
& k \ar[r]^{\tau_{V^\prime}} & \CF_{V^\prime} \ar[r] & {},
}
\]
where $V^\prime = (\res_{G,H}^*)^{-1}(V)$, the inverse image of $V$ 
under the restriction map.
\end{lemma}

\begin{proof}
The proof is a straightforward matter checking that the varieties 
are correct. 
\end{proof} 

In the case of $G$ an elementary abelian groups,
one of the main theorem in \cite{Ctriv} is the following.

\begin{thm} \label{thm:decomp}
Suppose that $G$ is an elementary abelian $p$-group having $p$-rank 
at least 3. Suppose that $V$ is a subvariety of $V_G(k)$ consisting 
of a single closed point. Let $kH$ be the image of a
flat map $\gamma:k[t_1, \dots, t_{r-1}]/(t_1^p, \dots, t_{r-1}^p) \to kG$
with the property that $V$ is not in $\res_{G,H}^*(V_H(k))$. 
Suppose that $Z = \alpha(t)$ where $\alpha:k[t]/(t^p) \to kG$
is a $\pi$-point whose equivalence class is the point in $V$. 
Then, the idempotent module $\CF_V$ has a
decomposition (as a direct sum of $kH$-modules)
\[
\CF_V = k \oplus P_0^{p-1} \oplus P_1
\oplus P_2^{p-1} \oplus P_3 \oplus \dots
\]
where
\[
\xymatrix{
\dots \ar[r] & P_2 \ar[r]^\partial & P_1 \ar[r]^\partial &
P_0 \ar[r]^\varepsilon & k \ar[r] &0
}
\]
is a projective $kH$-resolution of the trivial $kH$-module.
Multiplication by the element $Z$ is zero on the summand $k$.
For \ $m \in P_{2i-1}$, $i >0$,
\[
Zm \  = \ -(\partial(m), 0, \dots, 0) \in P_{2i-2}^{p-1}.
\]
For $m = (m_1, \dots, m_{p-1}) \in P_{2i}^{p-1}$,
\[
Zm = \begin{cases} - \varepsilon(m_{p-1}) + (0, m_1, \dots, m_{p-2})
\in k \oplus P_0^{p-1} & \text{ if } i = 0, \\
-\partial(m_{p-1}) + (0, m_1, \dots, m_{p-2}) \in 
P_{2i-1} \oplus P_{2i}^{p-1} &
\text{ if } i > 0.  \end{cases}
\]
The map $\tau_V: k \to \CF_V$ has image the summand $k$ in the
decomposition. Moreover, $\Hom_{kG}(k, \CF_V) = \sum_{i\geq 0}
H_i$ is a graded ring with
\[
H_i \cong  \begin{cases} k\tau_V(1) & \text{ for } i = 0, \\
\Hom_{kH}(k, P_i) & \text{ for } i >0 \end{cases}
\]
A homogeneous element  $\theta: k \to \CF_V$ lifts to a homomorphism
$\hat{\theta}: \CF_V \to \CF_V$ that is induced by a $kH$-chain map
$\theta_*:(P_*, \varepsilon) \to (P_*, \varepsilon)$ of the augmented
projective resolution to itself, that lifts $\theta$.
\end{thm}

\begin{proof}
Let $kC$ denote the image of $\alpha$. Because $\gamma$ is a flat 
map, the classes modulo
$\Rad^2(kG)$ of $\gamma(t_1), \dots, \gamma(t_{r-1})$ span a subspace
of $\Rad(kG)/\Rad^2(kG)$ of dimension $r-1$. The fact that $V$ is not in 
$\res_{G,H}^*(V_H(k))$ implies that the class modulo $\Rad^2(kG)$ of 
$Z$ is not in that subspace. That is, otherwise there would be a 
a $\pi$-point equivalent to $\alpha$ that factored through $\gamma$
violating the assumption on the varieties. Thus we have that 
$kG \cong kC \otimes kH$, is the group algebra of the direct 
product $C\times H$. Now we apply \cite[Theorem 6.2]{Ctriv},
where $V$ the point $[1, 0, \dots, 0]$ in $V_G(k)$ corresponding to $Z$. 
This gives the stated result.
That is, for this specific choice of $V$ and generators
$Z, \alpha(t_1), \dots, \alpha(t_{r-1}),$ 
the module $\CF_V$ has a decomposition as
described in \cite{Ctriv}. 
\end{proof}

We remark that changing the generators of $kG$, as we have done above,
does not preserve the Hopf algebra structure. However, as noted in
\cite[Remark 7.5]{Ctriv}, the structure of the idempotent modules does
not depend on the coalgebra structure.

\begin{thm} \label{thm:deeprad}
Assume the hypothesis of the previous theorem (Thm. \ref{thm:decomp}).
Suppose that $X = \alpha(t)$ where $\alpha:k[t]/(t^p) \to kG$ is a 
$\pi$-point not corresponding to the point $V$.
Assume that $\theta: k \to \CF_V$ is a homomorphism with 
the property that $\theta(1) \in X^{p-1}\CF_V$. Then,
the image of $\hat{\theta}: \CF_V \to \CF_V$ is contained in
$X^{p-1}\CF_V$.
\end{thm}

\begin{proof}
In the statement of Theorem \ref{thm:decomp}, the generators 
$Y_i = \alpha(t_i)$ of $kH$
can be chosen so that $Y_1 = X$ and $Y_2, \dots, Y_{r-1}$ are any elements
so that the classes modulo $\Rad^2(kG)$ of $Z, X, Y_2, \dots, Y_{r-1}$ 
form a basis for $\Rad(kG)/\Rad^2(kG)$.
Thus, $kH \cong kC \otimes kJ$ where
$kC \cong k[X]/(X^p)$ is the flat subalgebra of $kG$ generated by
$X$ and $kJ$ is the flat subalgebra  generated by $Y_2, \dots, Y_{r-1}.$
Let $(R_*, \varepsilon_1)$ be a minimal projective $kC$-resolution of $k_C$ and
$(Q_*, \varepsilon_2)$,  minimal projective $kJ$-resolution of $k_J$.
Then $R_i \cong kC$ for all $i \geq 0$. The minimal $kH$-resolution
of $k$ can be taken to be the tensor product of these two, so that
\[
P_n = \sum_{i= 0}^n R_i \otimes Q_{n-i}
\]

In the decomposition of $\CF_V$ given in Theorem \ref{thm:decomp}, the
element $\theta(1) \in X^{p-1}\CF_V \subseteq \sum_{n \geq 0} P_n$. Because
an assignment of chain maps to elements of $\Hom_{kG}(k, \CF_V)$ is
additive, it is sufficient to prove the theorem assuming that
$\theta(1) \in X^{p-1}P_n$ for some $n$. Indeed, we may assume that
$\theta(1) \in X^{p-1}(R_m \otimes Q_{n-m}) =
(X^{p-1}R_m) \otimes Q_{n-m})$ for some $m$ and $n$.
Because $R_m \cong kC$, we have that $\theta(1) = X^{p-1} \otimes u$ 
for some $u \in Q_{n-m}$.

Let $\varphi: k \to Q_{n-m}$ be given by $\varphi(1) = u$. Then
$\varphi$ lifts to a chain map
\[
\xymatrix{
\dots \ar[r] & Q_1  \ar[r] \ar[d]^{\varphi_1}  &
Q_0 \ar[r] \ar[d]^{\varphi_0} & 
k \ar[d]^{\varphi} \ar[r] & 0\\
\dots \ar[r] & \quad Q_{n-m+2} \ar[r] & \quad Q_{n-m+1} \ar[r] & 
\quad Q_{n-m} \ar[r] & \dots
}
\]

Let $\mu: R_0 \to X^{p-1}R_m$ be given by $\mu(1) = X^{p-1}$.
Now define $\theta_i: P_i \to P_{n+i+1}$ as the composition
\[
\xymatrix{
P_i \ar[r] & R_0 \otimes Q_i \ar[r] &
R_0/(X^{p-1}R_0) \otimes Q_i \ar[r]^{ \ \ \mu \otimes \varphi_i} &
X^{p-1}R_m \otimes Q_{n-m+i+1} \ar[r] & P_{n+i+1}
}
\]
The first map is projection onto the direct summand $R_0 \otimes Q_i$.
The second is the natural quotient. Then comes the chain map, 
and the fourth is the inclusion.

The task to finish the proof amounts to two straightforward exercises which
we leave to the reader. The first is to show that $\{\theta_i\}$ is
a chain map, and the second is to show that it lifts the map $\theta$.
\end{proof}

\begin{cor} \label{cor:restrict}
Assume the hypotheses and notation
of Theorems \ref{thm:decomp} and \ref{thm:deeprad}.
Let $\CI$ be the collection of all $\theta:k \to \CF_V$ such that
$\theta(1) \in X^{p-1}\CF_V$. Then under the correspondence
$\Homul_{kG}(k, \CF_V) \cong \Homul_{kG}(\CF_V, \CF_V)$, $\CI$ is
the kernel of the restriction map
$\Homul_{kG}(\CF_V, \CF_V) \to \Homul_{kH}(\CF_V, \CF_V)$.
In particular, $\CI$ is an ideal.
\end{cor}

\begin{proof}
The point is, in the notation of the last proof, that $X^{p-1}\CF_V
\subset P_*$. Thus, $\theta(1)$ and $\hat{\theta}(\CF_V)$ are in
$P_*$ which is free as a $kH$-module. Hence, the map $\theta$
factors through a $kH$-projective object and is zero on restriction
to $kH$.
\end{proof}

\begin{rem} \label{rem:indep}
We emphasize that in Theorem \ref{thm:decomp}, the choices of the flat map
$\gamma$ and also of the generators for $kH$ are arbitrary except that 
$kH$ should have rank $r-1$ and the condition on the varieties must be
satisfied. Similarly, in Theorem \ref{thm:deeprad},
any $\pi$-point $\alpha$ satisfying the 
desired conditions can be chosen. 
\end{rem}

\section{Endomorphisms of $\CF_V$}
Throughout this section, assume that $G = \langle g_1, \dots, g_r \rangle$
is an elementary abelian group of order $p^r$ for $r \geq 3$. We show
that if $V \subset V_G(k)$ is a closed subvariety of dimension $0$, then
the endomorphism ring of the idempotent module $\CF_V$ in the stable category
has a unique maximal ideal that is nilpotent and has codimension one.
We assume the notation of the previous section.

We prove the following result in more generality than is actually needed
in the later development.

\begin{prop} \label{prop:disjointV}
Suppose that $V_1$ and $V_2$ are disjoint subvarieties of $V_G(k)$. 
Let $kH$ be the image of a flat map 
$\gamma:k[t_1, \dots, t_s]/(t_1^p, \dots, t_s^p) \to kG$, for $s \geq 2$.
Let $X = \gamma(t_1)$ and let $kJ$ be the flat subalgebra of 
$kG$ generated by $\gamma(t_2), \dots, \gamma(t_s)$. 
Assume that we have the following two conditions.
\begin{enumerate}
\item $V_1 \subseteq  \res_{G,J}^*(V_J(k)).$
\item $V_2 \cap \res_{G,H}^*(V_H(k)) = \emptyset$.
\end{enumerate}
Suppose that $\varphi: k \to \CF_2 = \CF_{V_2}$ is a homomorphism 
such that $\varphi(1) \in X^{p-1}\CF_2$. Then $\varphi$ extends to a 
homomorphism $\psi: \CF_1 = \CF_{V_1} \to X^{p-1}\CF_2$. That is, we have a 
commutative diagram 
\[
\xymatrix{
k \ar[r]^{\tau_1} \ar[d]_{\varphi}  & \CF_1 \ar[dl]_\psi \ar[d] \\
X^{p-1}\CF_2 \ar[r]^{ \quad \iota}  & \CF_2
}
\]
where $\iota$ is the inclusion.
\end{prop}

\begin{proof}
By Condition (2), $(\CE_2)_{\downarrow H}$ is free as a $kH$-module. This 
implies that $(\CF_2)_{\downarrow H} \cong k \oplus P$, where $P$ is a 
free $kH$-module. Thus, $\Homul_{kH}(k, (\CF_2)_{\downarrow H})$ has 
dimension one, and the fact that $\varphi(1) \in X^{p-1}\CF_2$ means that 
$\varphi$ factors through a projective $kH$-module, namely $P$. This 
follows because $\varphi(1) \in \Rad(kH)\CF_2 \cap \Soc(\CF_2) \subset P$.

Note that $X \not\in kJ$, and hence the restriction of  
$\CE_1$ to the subalgebra generated by $X$ is a free module. Moreover, 
$(X^{p-1}\CF_2)_{\downarrow H} = X^{p-1}P$ is free as a $kJ$-module. 
Consequently, $\CV_G(\CE_1) \cap \CV_G(X^{p-1}\CF_2) = \emptyset$, and
the composition 
\[ 
\xymatrix{ 
\CE_1 \ar[r]^{\sigma_1} & k \ar[r]^{\varphi \qquad}  & X^{p-1}\CF_2
}
\]
is the zero map in the stable category 
by Lemma \ref{lem:Elocal}. The existence of the map $\psi$ is implied
from the distinguished triangle. 
\end{proof}

\begin{cor}  \label{cor:extend}
Suppose that $G$ is an elementary abelian $p$-group having rank
at least 3. Suppose that $V_1, V_2 \subset V_G(k)$ are closed subvarieties
each consisting of a single point. Let 
$\beta: k[t_1,t_2]/(t_1^p,t_2^p) \to kG$ be flat map such that the 
following hold. For notation, let $kH$ be the image of $\beta$. 
\begin{enumerate}
\item The class of the $\pi$-point $\alpha:k[t]/(t^p) \to kG$ 
with $\alpha(t) = \beta(t_1)$ is in $V_1$. 
\item $V_2 \not\subset \res_{G,H}^*(V_H(k))$. 
\end{enumerate}
Let $X = \beta(t_2)$. Suppose that $\varphi: k \to \CF_2 
= \CF_{V_2}$ is a map such that $\varphi(1)\in X^{p-1}\CF_2$. 
Then $\varphi$ extends to a map $\theta: \CF_1 \to \CF_2$ such that 
$\theta(\CF_1) \subseteq X^{p-1}\CF_2$.  
\end{cor} 

\begin{proof}
Let $kJ$ be the image of $\alpha$. Then the conditions of 
Proposition \ref{prop:disjointV} are satisfied and the corollary follows.
\end{proof}

We can now prove the main theorem of the section. 

\begin{thm} \label{thm:elemain}
Suppose that $G$ is an elementary abelian $p$-group having rank $r \geq 3$.
Let $V \subset V_G(k)$ be a closed subset consisting of a finite
number of closed points. Let $kH$ be the image of a flat map 
$\gamma:k[t_1, \dots, t_s]/(t_1^p, \dots, t_s^p) \to kG$ 
such that $s \geq 2$ and $\res_{G, H}^*(V_H(k)) \cap V = \emptyset$. Let 
$\CI$ be the kernel of the restriction $\Endul_{kG}(\CF_V) \to 
\Endul_{kH}(\CF_V)$. Then $\CI$ is an ideal of codimension one in 
 $\Endul_{kG}(\CF_V)$, and $\CI^2 = \{0\}$. Thus $\CI$ is the unique 
maximal ideal in $\Endul_{kG}(\CF_V)$.  
\end{thm}

\begin{proof}
We write $V = \cup_{i=1}^{n} V_i$ where each $V_i$ is a closed point in 
$V_G(k)$. For each $i$, let $\alpha_i:k[t]/(t^p) \to kG$ be a $\pi$-point 
corresponding to the closed point in $V_i$. For $i \neq j$ let $kH_{i,j}$
be the image of the flat map $\beta_{i,j}: k[t_1,t_2]/(t_1^p,t_2^p) \to kG$ 
with $\beta_{i,j}(t_1) = \alpha_i(t)$ and $\beta_{i,j}(t_2) = \alpha_j(t)$. 
We note that for all $i , j$, the intersection 
\[
\res_{G, H_{i,j}}^*(V_{H_{i,j}}(k)) \cap \res_{G, H}^*(V_{H}(k))
\]
either contains only one point or is empty. 
Let $\alpha:k[t]/(t^p) \to kG$ be a $\pi$-point 
that factors through $\gamma$, but is not in 
$\res_{G, H_{i,j}}^*(V_{H_{i,j}}(k))$ for any pair $i,j$ with 
$1 \leq i < j \leq n$. 

By Proposition \ref{prop:nointersect}, the idempotent module $\CF_V$
is the pushout of the system $\{\tau_i: k \to \CF_i$\}, where $\CF_i = 
\CF_{V_i}$. For each $i$, there is a homomorphism $\nu_i: \CF_i \to \CF$, 
such that the compositions $\nu_i\mu_i$ coincide.  
From the conditions on the choice of $kH$, we see that 
for each $i$, the restriction of $\CF_i$ to $kH$ 
has the form $k \oplus P_i$ where $P_i$ is a projective $kH$-module. 
In addition, the component isomorphic to $k$ is generated by $\tau_i(1)$.
Thus, the restriction $(\CF_V)_{\downarrow H} \cong k \oplus \sum P_i$. 
We see that $\CI$ is the subspace of $\Homul_{kG}(k, \CF_V)$
spanned by all $\varphi: k \to \CF_V$ with 
$\varphi(1) \in X^{p-1} \CF_V$.  

For any $j = 1, \dots, t$, let $\CI_j$ be the set of all $\varphi \in \CI$
such that $\varphi(1) \in X^{p-1}P_j$. Thus $\CI$ is the direct sum of 
the subspaces $\CI_j$. 

Throughout the proof, we make the identification 
$\Homul_{kG}(k, \CF_V) \cong \Homul_{kG}(\CF_V, \CF_V)$. If we are 
given two elements $\varphi_1$ and $\varphi_2$ in $\Homul_{kG}(k, \CF_V)$,
their product is obtained by first finding lifts 
$\hat{\varphi_i}: \CF_V \to \CF_V$, for $i = 1,2$, taking the composition and
composing with the map $\tau: k \to \CF_V$. Note that
any lift will serve the purpose. Our aim is to show that if 
$\varphi_1, \varphi_2 \in \CI$, then the product is zero. Without loss
of generality we may assume that $\varphi_i \in \CI_{j_i}$ for some
$1 \leq j_i \leq t$. 

Letting $j = j_1$, there is a $\phi_1: k \to X^{p-1}P_j \subset P_j$
such that $\nu_j\phi_1 = \varphi_1: k \to \CF_V$.
By Corollary \ref{cor:extend}, for any $i \neq j$ there is an extension
$\theta_i: \CF_i \to \nu_j(X^{p-1}\CF_j)$ extending $\phi_1$. Likewise,
by Theorem \ref{thm:deeprad}, there is such an extension also in the 
case that $i = j.$ Thus, for every $i = 1, \dots, t$, there is an 
extension $\hat{\theta}_i: \CF_i \to \CF_V$ of 
$\varphi_1$ with the property that 
$\hat{\theta}_i(\CF_i) \subseteq X^{p-1}\CF_V$. 

The universal property of pushouts, now guarantees that there is a map
$\hat{\varphi_1}: \CF_V \to \CF_V$ such that, for every $i$, the diagram 
\[
\xymatrix{
k \ar[r]^{\phi_1} \ar[d]_{\tau} \ar[dr]^{\varphi_1}  & 
\CF_j \ar[d]^{\nu_j} \\
\CF_V \ar[r]_{\hat{\varphi_1}} & \CF_V
}
\]
commutes and $\hat{\varphi_1}(\CF_V) \subseteq X^{p-1}\CF_V$.
There is a similar extension $\hat{\varphi_2}: \CF_V \to \CF_V$ with the 
same property. 

Now, we see that $\hat{\varphi_2}(\hat{\varphi_1}(\CF_V) \subseteq
\hat{\varphi_2}(X^{p-1}\CF_V) \subseteq X^{2(p-1)}\CF_V = \{0\}$.
Hence, the product of any two elements in $\CI$ is zero.
We notice that $\Endul_{kH}(\CF_V) \cong k$, implying 
that $\CI$ is a maximal ideal. Because $\CI^2 = \{0\}$,
any element of $\Endul_{kG}(\CF_V)$ that is not in $\CI$ is invertible. 
\end{proof}

\section{The general case}
The aim of this section is to extend the conclusion of Theorem 
\ref{thm:elemain} to a more general finite group $G$. 
The arguments in the proofs depend on the fact (see Remark \ref{rem:quillen})
that any prime
ideal in $\hgs$ contains the kernel of a restriction to an elementary
abelian $p$-subgroup of $G$. For this reason, the main results of the section
do not extend to general finite group schemes. Throughout the section
we assume the following. 

\begin{hyp} \label{hyp}
suppose that $G$ is a finite group whose maximal
elementary abelian $p$-subgroups all have $p$-rank at least three.
Let $V$ be a closed subvariety of $V_G(k)$, that
is a union of a finite collection of closed points. 
\end{hyp}

We make the identification 
$\Homul_{kG}(k, \CF_V) \cong \Homul_{kG}(\CF_V, \CF_V)
= \Endul_{kG}(\CF_V)$, as before. 
The ideal, that we are interested in, is the following.  

\begin{defi} \label{def:maxideal}
Assume that \ref{hyp} holds. 
Suppose that $E$ is an elementary abelian $p$-subgroup of $G$ 
with order $p^r$ for $r \geq 3$.
Let $kH$ be the image of a flat map 
$\gamma: k[t_1, \dots, t_{s}]/(t_1^p, \dots, t_{s}^p) \to kE$, 
such that $1 \leq s < r$ and 
\[
\res^*_{G,H}(V_H(k)) \cap V = \emptyset.
\]
Let $\CI \subset \Endul_{kG}(\CF_V)$ be the kernel of the restriction 
map $\Endul_{kG}(\CF_V) \to \Endul_{kH}(\CF_V)$. 
\end{defi}

Given $E$, the existence of $H$ follows easily from the geometry. Note that 
$\CI$ is an ideal, because if a $kG$-homomorphism factors through a 
$kH$-projective module, then so does its composition with any other 
homomorphism. 

\begin{prop} \label{prop:notdepend}
Assume that \ref{hyp} holds. 
The ideal $\CI$ does not depend on the choice of 
$E$ or $H$, as long as the above conditions are satisfied.
\end{prop}

\begin{proof}
The first thing to notice is that the restriction of the module $\CE_V$ to
$kH$ is projective, and hence, $(\CF_V)_{\downarrow H} \cong k \oplus P$ 
where $P$ is a projective module by Lemma \ref{lem:restr-idem}.
Then, the independence of the choice of $kH$ in $kE$ follows from
Theorem \ref{thm:elemain}. 
 
Now suppose that $E_1$ and $E_2$ are elementary abelian $p$-subgroups of 
$G$ such that $F = E_1 \cap E_2$ has $p$-rank at least $2$. Let $\CI_j$ be
the ideal defined by $E_j$ as above for $j = 1,2$. We claim that 
$\CI_1 = \CI_2$. The reason is that there must be a $\pi$-point 
$\alpha: k[t]/(t^p) \to kF \subset kG$ with the property that 
the equivalence class of $\alpha$ does not correspond to any point 
of $(\res_{G,F}^*)^{-1}(V)$, the inverse image $V$ under the restriction 
map $\res_{G,F}^*: V_F(k) \to V_G(k)$. Let $H$ be the image of $\alpha$.
Then $kH$ is a flat subalgebra of both $kE_1$ and $kE_2$, which 
satisfies the condition of Definition \ref{def:maxideal}.
So both $\CI_1$ and $\CI_2$ are the kernel of the restriction to $kH$. 

Now suppose $G$ is a $p$-group and that $E$, $E^{\prime}$ are any two 
elementary abelian $p$-subgroups having 
$p$-rank at least 3. By the argument of Alperin (see the bottom of page 
8 to top of page 9 of \cite{Alp}), there is a chain $E = F_1, F_2, 
\dots, F_m = E^\prime$ of elementary abelian $p$-subgroups of $G$
such that $F_i \cap F_{i+1}$ has $p$-rank at least 2. Thus by an 
easy induction and the previous paragraph, the ideal $\CI$ is independent
of the choice of $E$.

If $G$ is not a $p$-group we need only notice that if $E_1$ and $E_2$ 
are conjugate elementary abelian $p$-subgroups, then the ideals 
$\CI_1$ and $\CI_2$ must be the same. Thus, we may assume that any 
two elementary abelian $p$-subgroup are in the same Sylow $p$-subgroup.
In such a case the same proof as above works. 
\end{proof}

\begin{thm} \label{thm:general}
Assume that the conditions of $\ref{hyp}$ hold. 
Let $\CF_V$ be the idempotent $\CF$-module corresponding to 
$V$. Then $\Endul(\CF_V)$ is a local ring whose unique maximal ideal $\CI$ is 
nilpotent. Moreover, there is a number $B$ depending only on $G$ and $p$
such that the nilpotence degree of $\CI$ is at most $B$. 
\end{thm}

\begin{proof}
The proof is an easy consequence of the above Proposition, 
Theorem \ref{thm:elemain} and Theorem 2.5
of \cite{CQ}.  The last mentioned 
theorem can be interpreted as saying that there
is number $N$, depending only on $G$ and $p$, such that for any 
sequence $M_0, \dots, M_n$ of $kG$-modules and maps 
$\theta_i \in \Homul_{kG}(M_{i-1}, M_i)$,
$1 \leq i \leq n$, such that $n \ge N$ and 
$\res_{G, E}(\theta_i) = 0$ for every
elementary abelian subgroup $E$ of $G$, the composition
$\theta_n \cdots \theta_1= 0$. In the case that $M_i = \CF_V$ for all $i$
and $B = 2N$, choose elements $\theta_i \in \CI$. Then, 
for every $i$, the restriction of the product
$\theta_{2i}\theta_{2i-1}$ to 
every elementary abelian $p$-subgroup of $G$ vanishes, by Theorem
\ref{thm:elemain}. It follows that 
$\theta_n \cdots \theta_1= 0$ if $n \geq B$. 
\end{proof}

\section{Restrctions} \label{sec:restrict}
The aim of the section is prove a few facts about the restrictions of the
endomorphism rings. The first result is known, but perhaps has not been 
written down. 

\begin{lemma} \label{lem:rest01}
Suppose that $G$ is an elementary abelian $p$-group of order $p^r >1$.
Let $kH \neq kG$ be a flat subalgbra of $kG$. Then for all 
$n < 0$, we have that 
\[
\res_{G,H}: \HHHH^n(G,k) \to \HHHH^n(H,k)
\]
is the zero map. 
\end{lemma}

\begin{proof}  Let $\gamma: k[t_1, \dots, t_s]/(t_1^p, \dots, t_s^p) \to kG$
be a flat map whose image is $kH$. Then the classes modulo $\Rad^2(kG)$ of 
$\gamma(t_1), \dots, \gamma(t_s)$ are $k$-linearly independent in 
$\Rad(kG)/\Rad^2(kG)$. Let $b_{s+1}, \dots, b_r$ be elements that 
are chosen so that the classes of $\gamma(t_1), \dots, \gamma(t_s), 
b_{s+1}, \dots, b_r$ form a basis for $\Rad(kG)/\Rad^2(kG)$. 
Let $kJ$ be the flat subalgebra generated by $b_{s+1}, \dots, b_r$ 
so that $kG \cong kH \otimes kJ$ (see Lemma \ref{lem:complement}).

The Tate cohomology group $\HHHH^n(G,k)$ is isomorphic to 
$\Hom_{kG}(k,P_{n+1})$, where $P_*$ is a minimal projective $kG$-resolution
of $k$. The restriction is the map 
$\psi_*: \Hom_{kG}(k,P_{n-1}) \to \Hom_{kG}(k,Q_{n-1})$, where $Q_*$ is 
a minimal $kH$-projective resolution of $k$, and $\psi: P_* \to Q_*$ is a 
$kH$-chain map. Because of the decomposition $kG \cong kH \otimes kJ$, we 
may assume that $Q_*$ is a complex of $kG$-modules on which $kJ$ acts
trivially and that $\psi$ is a $kG$-chain map. However, $P_{n-1}$ is a
free $kG$-module, implying that any map $\zeta:k \to P_{n-1}$ has its 
image in $\Rad(kJ)P_{n-1}$. Because $kJ$ acts trivially on $Q_{n-1}$,
the image of $\zeta$ is in the kernel of $\psi$. 
\end{proof} 

\begin{thm} \label{thm:rest02}
Suppose that $G$ is an elementary abelian $p$-group of 
$p$-rank $r \geq 3.$  
Let $V$ be a subvariety of $V_G(k)$ that
is a union of a finite collection of closed points.
Suppose that $kH \neq kG$ is a flat subalgebra of $kG$. 
Then the maximal ideal $\CI \subseteq \Endul_{kG}(\CF_V)$ is 
the kernel of the restriction map 
$\res_{G,H}: \Endul_{kG}(\CF_V) \to \Endul_{kH}((\CF_V)_{\downarrow H})$.
\end{thm}

\begin{proof}
We write $V = \cup_{i=1}^{n} V_i$ where each $V_i$ is a closed point in
$V_G(k)$. Recall that by Proposition 
\ref{prop:nointersect}, the idempotent module $\CF_V$
is the pushout of the system $\{\tau_i: k \to \CF_i$\}. Here, $\CF_i =
\CF_{V_i}$ and for each $i$, there is a homomorphism $\nu_i: \CF_i \to \CF$,
such that the compositions $\nu_i\tau_i$ coincide.

Let $\gamma_i:k[t_1, \dots, t_{r-1}]/(t_1^p, \dots, t_{r-1}^p) \to kG$
with image $kJ$ such that 
\[
\res_{G,J}^*(V_J(k)) \cap V = \emptyset
\]
For each $i$, we have that the restriction of $\CF_i = \CF_{V_i}$ to $kJ$ 
has the form $\CF_i \cong k \oplus Q_i$ where $Q_i$ is a projective 
$kJ$-module. It follows that 
\[
\CF_{\downarrow J} \ \cong \ k \oplus \nu_1(Q_1) \oplus \dots \oplus \nu_n(Q_n).
\]
If $\zeta: k \to \CF_V$ is in $\CI$, then 
$\zeta(1) \in \sum \nu_i(Q_i)$ by Theorem 
\ref{thm:elemain} and Corollary \ref{cor:restrict} (see also Remark
\ref{rem:indep}). 

Hence, for the remainder of the proof we fix an 
element $\zeta \in \CI$, and without loss of generality, we may assume
that $\zeta(1) \in \nu_i(Q_i)$ for some fixed $i$. Our object is to show that
$\zeta$ factors through a $kH$-projective module. Notice that $\zeta$ 
must factor through $\nu_i: \CF_i \to \CF_V$. Consequently, it is 
sufficient to prove the theorem in the case that $\CF_V = \CF_i$. 
That is, we may assume that $n=1$, $V = V_i$. 

There are two cases to consider. First assume that $\res_{G,H}^*(V_H(k))$
does not contain the point $V$. In this case $(\CE_V)_{\downarrow H}$ 
is projective, and hence $(\CF_V)_{\downarrow H} \cong k$ in the stable 
category. In this case, the restriction of $\CI$ to $kH$ is zero and 
we are done. 

Next,  we assume that $V \subset \res_{G,H}^*(V_H(k))$. There is a 
$\pi$-point $\alpha: k[t]/(t^p) \to kH \subset kG$ whose equivalence 
class is the one closed point in $V$. Let $kC$ be the 
image of $\alpha$. The flat subalgebra $kH$ has a maximal flat subalgebra
$kL$ such that $kH \cong kC \otimes kL$. There is a flat subalgebra
$kD$ such that $kG \cong kH \otimes kD$
(see Lemma \ref{lem:complement}). We may assume that the subalgebra 
$kJ$ has the form  $kJ = kL \cdot kD \cong kL \otimes kD$,
because  $V$ is not contained in $\res_{G,J}^*(V_J(k))$. Therefore,  
by Theorem \ref{thm:elemain}, $(\CF_V)_{\downarrow J} \cong k \oplus P$,
where $P$ is the sum of the terms (with multiplicities) 
of a minimal augmented $kJ$-projective 
resolution of $k$ such that element $\alpha(1)$ acts
as the boundary homomorphism on 
the augmented complex. Likewise, the restriction of $\CF_V$ to $kL$ 
has the form $(\CF_V)_{\downarrow L} \cong k \oplus Q$ is the sum of the 
terms (with multiplicities) of an augmented 
minimal $kL$-projective resolution of $k$.
As in the proof of Lemma \ref{lem:rest01}, the restriction map in the 
stable category is given by a chain map of augmented complexes which 
can be shown to be a $kG$-homomorphism. Because $\zeta \in \CI$, we have that
$\zeta(1) \in P$, and as in the proof of Lemma \ref{lem:rest01}, the
chain map takes $\zeta(1)$ to zero. 

This proves that $\CI$ is in the kernel of the restriction to 
$\Endul_{kH}(\CF_V)$. The fact that it is the kernel is a consequence 
of its maximality.  
\end{proof}

\section{Finite generation}
In this final section we address the issue of 
the finite generation of the endomorphism
rings. Suppose that $V$ is a closed subvariety of $V_G(k)$. Let $\CM_V$ 
be the subcategory of all $kG$-modules whose support variety is contained
in $V$, and let $\CC_V$ be the localization of $\Stmodg$ at $\CM_V$. 
Tensoring with $\CF_V$ is the localization functor, and 
for $kG$-modules $M$ and $N$, we have that 
\[
\Hom_{\CC_V}(M,N) \cong \Homul_{kG}(M \otimes \CF_V, N \otimes \CF_V)
\]
is a module over $\End_{\CC_V}(k) \cong \Endul_{kG}(\CF_V)$. The question is 
whether it is a finitely generated module? 

The question is most relevant 
in the case that $M$ and $N$ are finitely generated modules (compact objects
in $\Stmodg$). Of course, if the support varieties of the finitely generated
objects $M$ and $N$ are both disjoint from $V$ then the finite generation
is obvious. This is because, in such a case, $\CE_V \otimes M$ and 
$\CE_V \otimes N$ are projective modules, and hence $\CF_V \otimes M \cong M$
and $\CF_V \otimes N \cong N$ in the stable category. Also, if  
the support variety of either $M$ or $N$ is contained in $V$, then the 
finite generation is also clear. In other situations
some proof is required. We show that the answer is not unlike the 
answer to the question of the finite generation of $\End_{\CC_V}(k)$. 

The proof of the following is straightforward generalization of a well
used argument. 

\begin{thm} \label{thm:codim1}
Suppose that $\zeta \in \HHH^d(G,k)$ is a non-nilpotent element. Let 
$V = V_G(\zeta)$, the collection of all homogeneous prime ideals that 
contain $\zeta$. If $M$ and $N$ are finitely generated $kG$-modules, then 
$\Hom_{\CC_V}(M,N)$ is finitely generated as a module over 
$\End_{\CC_V}(k)$. 
\end{thm}

In fact, what we want to show is the following. 

\begin{lemma} \label{lem:codim1}
Assume the hypothesis of Theorem \ref{thm:codim1}. Then 
\[
\Hom_{\CC_V}(M, N) \cong (\Ext_{kG}^*(M,N)[\zeta^{-1}])_0.
\]
That is, viewing $\Ext_{kG}^*(M,N)$ as a module over $\hgs$, we 
invert the action of $\zeta$ and take the zero grading.
\end{lemma} 

\begin{proof}
Suppose that in $\CC_V$, we have a morphism 
\[
\xymatrix{
\phi = \nu\mu^{-1}: & M & L \ar[l]_\mu \ar[r]^\nu & N
}
\]
for some $kG$-module $L$ such that the third object $X$ in the triangle
of $\mu$ is in $\CM_V$. This implies that for $n$ sufficiently large, 
the element $\zeta^n$ annihilates the cohomology of $X$. Consequently, 
we have as in the diagram 
\[
\xymatrix{
& \Omega^{dn}(M) \ar[d]^{\zeta^n} \ar@{-->}[dl]_{\theta} \\
L \ar[r]^{\mu} & M \ar[r]^\gamma & X 
}
\]
that the composition $\gamma\zeta^n = 0$. This implies the existence of the 
map $\theta$, and we have that $\phi = \nu\mu^{-1} = (\nu\theta)\zeta^{-n}$,
where $\nu\theta$ represents an element in $\Ext_{kG}^{nd}(M,N)$ and any
other representative defines the same element of $\Hom_{\CC_V}(M, N)$.
Likewise, the class of $\theta$ in $\Ext_{kG}^{nd}(M,M)$ is the class of
the cocycle $\theta \otimes 1: \Omega^{nd}(k) \otimes M \to k \otimes M$,
and all representatives of this class define the same element of
$\Ext_{kG}^{nd}(M,M)$. 
\end{proof}

\begin{proof}[Proof of Theorem \ref{thm:codim1}]
We now use the fact that $\Ext_{kG}^*(M,N)$ is finitely generated as a 
module over $\hgs \cong \Ext_{kG}^*(k,k)$. Let $\gamma_1, \dots,
\gamma_s$ be a set of homogeneous generators.
If $A = \sum_{n\geq 0} \Ext_{kG}^{nd}(k,k)$, then by elementary commutative 
algebra $\hgs$ is finitely generated over $A$.
Suppose that $\beta_1, \dots, \beta_r$ is a set of homogeneous generators.
Let $B = \sum_{n\geq 0} \Ext_{kG}^{nd}(M,N)$. If $\theta \in B$ is a 
homogeneous element then $\theta = \sum_{i=1}^s \gamma_i\mu_i$ for 
$\mu_i \in \hgs$. But then, for each i, 
$\mu_i = \sum_{j=1}^r \beta_j\alpha_{ij}$, for $\alpha_{ij} \in A$. 
Hence, 
\[
\theta = \sum_{i, j = 1,1}^{s,r} \gamma_i\beta_j\alpha_{ij}.
\]
That is, we see that the products $\gamma_i\beta_j$ generate $B$ as a module 
over $A$. Note that we need really only consider those whose 
degree is a multiple of $d = \text{Degree}(\zeta)$. 
Now we have that $\Hom_{\CC_V}(M, N)$ is generated by elements having 
the form $\gamma_i\beta_j\zeta^{-r}$, where 
$rd = \text{Degree}(\gamma_i\beta_j)$. 
\end{proof}
 
In the other direction we have the following. The example is far
from general, but perhaps the reader can see how other examples can 
be constructed. 

\begin{thm} \label{thm:nofingen}
Suppose that $V \subset V_G(k)$ is a closed subvariety such that there
is an elementary abelian $p$-subgroup $E$ with
\[
\res_{G,E}^*(V_E(k)) \cap V
\]
a nonempty finite set of closed points. Assume that $\vert E \vert \geq p^3$
and that $E$ has a subgroup $F$ with $\vert F \vert = p^2$ and
$\res_{G,F}^*(V_F(k)) \cap V \neq \emptyset$.  Let $M = k_F^{\uparrow G} 
= kG \otimes_{kF} k_F$ be the induced module.
Then $\End_{\CC_V}(M)$ is not finitely generated as a module
over $\End_{\CC_V}(k)$.
\end{thm}

\begin{proof}
First we note that, in the category $\CC_V$, $M \cong \CF_V \otimes M$.
By Frobenius Reciprocity,
\[
\CF_V \otimes M \cong \CF_V \otimes k_F^{\uparrow G} \cong
(\CF_V)_{\downarrow F})^{\uparrow G}
\]

By Lemma \ref{lem:restr-idem}, $(\CF_V)_{\downarrow F} \cong 
\CF_{V^\prime}$ in the stable category where 
$V^\prime = (\res_{G,F}^*)^{-1}(V)$. From the hypothesis, we know
that $V^\prime$ consists of a finite set of closed points and is 
not equal to $V_F(k)$. 

By the Eckmann-Shapiro Lemma, we have the usual adjointness:
\[
\Homul_{kG}(\CF_V \otimes k_F^{\uparrow G}, \CF_V \otimes k_F^{\uparrow G})
\cong 
\Homul_{kF}((\CF_V)_{\downarrow F},
(\CF_V \otimes k_F^{\uparrow G})_{\downarrow F})
\]
\[
\cong
\Homul_{kF}((\CF_{V^\prime}),
((\CF_{V^\prime})^{\uparrow G})_{\downarrow F})
\]
Now notice that $((\CF_{V^\prime})^{\uparrow G})_{\downarrow F}$
has a direct summand isomorphic to $\CF_{V^\prime}$. That is, 
\[
(\CF_{V^\prime})^{\uparrow G} \cong \sum_{g \in G/F} g \otimes \CF_{V^\prime}
\]
as $k$-vector spaces. Here, the sum is over a complete set of left
coset representatives of $F$ in $G$.
The subspace $1 \otimes \CF_{V^\prime}$ 
is a $kF$-submodule and a direct summand.  
This also follows from the Mackey Theorem. The 
implication is that $\Homul_{kF}((\CF_{V^\prime}),
((\CF_{V^\prime})^{\uparrow G})_{\downarrow F})$ has a direct sumand
isomorphic to $\Endul_{kF}(\CF_{V^\prime})$. We have seen in earlier 
sections of this paper that this has infinite $k$-dimension. 

The action of $\End_{\CC_V}(k)$ on $\End_{\CC_V}(M)$ is given by
\[
\xymatrix@-.3pc{
\Homul_{kG}(\CF_V, \CF_V) \otimes
\Homul_{kG}(\CF_V \otimes k_F^{\uparrow G}, \CF_V \otimes k_F^{\uparrow G})
\ar[d] \\
\Homul_{kF}((\CF_V)_{\downarrow F}, (\CF_V)_{\downarrow F}) \otimes
\Homul_{kF}((\CF_V)_{\downarrow F},
(\CF_V \otimes k_F^{\uparrow G})_{\downarrow F})
\ar[d] \\
\Homul_{kF}((\CF_V)_{\downarrow F},
(\CF_V \otimes k_F^{\uparrow G})_{\downarrow F})
}
\]
where the first arrow is the isomorphism given by the Eckmann-Shapiro 
Lemma and the second is composition. The Eckmann-Shapiro Lemma is easily 
seen to hold in the stable category. 

The main point of the proof is that
in applying the Lemma, the action of $\End_{\CC_V}(k)$ factors through
the restriction map to $kF$. However, the restriction map is transitive,
and hence must factor through the restriction to $kE$. By Theorem 
\ref{thm:rest02}, the restriction of the maximal ideal $\CI$ in 
$\End_{\CC_V}(k)$ is zero. That is, the image of the restriction 
of $\End_{\CC_V}(k) = \Homul_{kG}(\CF_V, \CF_V)$ to 
$\Homul_{kF}((\CF_V)_{\downarrow F}, (\CF_V)_{\downarrow F})$
is the identity subring $k$. 

It follows from the above that in order for 
$\End_{\CC_V}(M)$ to be finitely generated over $\End_{\CC_V}(k)$,
it must be finite dimensional. However, we have already noted that 
this is not the case.
\end{proof}

\section{Examples} \label{sec:exm}
We end the paper with a couple of examples and a theorem on the structure
of the idempotent modules. For the first example and most of the section
suppose that $G = SL_2(p^n)$ for $n > 2$, and let $k$ be an algebraically
closed field of characteristic $p$.

Let $a$ be a generator for the multiplicative group $\bF_{p^n}^\times$.
The Borel subgroup $B$ of $G$ is generated by the elements 
\[
t = \begin{bmatrix} a & 0 \\ 0 & a^{-1} \end{bmatrix} \quad \text{and} \quad   
x_i = \begin{bmatrix} 1 & a^i \\ 0 & 1 \end{bmatrix} \quad \text{for} \quad
i = 0, \dots, n-1.
\]
Then $S = \langle x_1, \dots, x_n \rangle$ is a Sylow $p$-subgroup and 
$B$ is its normalizer in $G$.  

The variety $V_S(k) \cong \bP^{n-1}$, projective $(n-1)$-space. The group 
$B$ acts on $S$ by conjugation and hence also on $V_S(k)$. The action of 
$T = \langle t \rangle$ is given by the relation 
\[
\begin{bmatrix} b & 0 \\ 0 & b^{-1} \end{bmatrix}
\begin{bmatrix} 1 & u \\ 0 & 1 \end{bmatrix}
\begin{bmatrix} b^{-1} & 0 \\ 0 & b \end{bmatrix} \quad 
= \quad \begin{bmatrix} 1 & ub^2 \\ 0 & 1 \end{bmatrix}
\]
for $u$ in $\bF_{p^n}$ and $b$ in $\bF_{p^n}^\times$. The thing to notice
is that if $b^2$ is in the prime field $\bF_p$, then this element of $T$ 
operates on $S$ (viewed as an $\bF_p$-vector space $S \cong (\bZ/(p))^n$)
by a scalar matrix with diagonal entries equal to $b^2$. This implies that
the element of $T$ acts trivially on the projective space $V_S(k)$. With 
this in mind,  let
\[
d \ = \ \begin{cases} p-1 & \text{if $p=2$ or $n$ is odd}, \\
2(p-1) & \text{otherwise}. \end{cases}
\]
It is easily checked that $d$ is the order of the subgroup of $T$ that 
acts trivially on $V_G(S)$. Let $m = (p^n-1)/d$, and let $D$ be the 
subgroup of $B$ generated by $c = t^{m}$ and $S$.  

Choose $W$ to be any subvariety of $V_S(k) = V_D(k)$ 
that consists of a single point whose 
stabilizer in $T = \langle t \rangle$ is generated by $c$. 
Let $V = \res_{B,S}^*(W)$. Then the inverse image of $V$ under the 
restriction map $\res_{B, S}^*$ is the union of the points in the orbit
of $W$ under the action of $T$. Let $V = V_0, \dots, V_{m-1}$ be the 
images under $V$ of this action. Let $\CF_V$ be the idempotent 
$kD$-module corresponding to $V$.

The induced module $\CF_V^{\uparrow B} = kB \otimes_{kD} \CF_V$
has the form $\sum_{j=0}^{m-1} \CF_{V_j}$ on restriction to $D$. We may 
assume that $\CF_{V_j} = t^{j} \otimes \CF_V$ in this context. Thus we 
have maps $t^j \otimes \tau_V: t^j \otimes k \to t^j \otimes \CF_V$. 
That is, $T$ acts on this system and also acts on the pushout that is 
obtained by identifying the images of the maps $t^j \otimes \tau_V$.
Explicitly, let $N$ be the submodule of 
$k_D^{\uparrow B} \cong kB \otimes_{kD} k$ generated by 
$1 \otimes 1 - t \otimes 1$. This is a $kB$-submodule of dimension 
$m-1$. Let $N^\prime$ be its image in $\CF_V^{\uparrow B}$, that is 
the submodule generated by $1 \otimes \tau_V(1) - t \otimes \tau_V(1)$.
Then we have a triangle
\[
\xymatrix{
{} \ar[r] & k_D^{\uparrow G}/N \ar[r] & \CF_V^{\uparrow B}/N^\prime \ar[r] &
\ar[r] \CF_V^{\uparrow B}/k_D^{\uparrow G} \ar[r] & {}
}
\]
Next, a check of the varieties can be done at the level of the Sylow 
$p$-subgroup $S$. In particular, the variety of $\CE_V$ is $\{V\}$, while 
that if $\CF_V$ is $\CV_G(k) \setminus \{V\}$. Thus, by the tensor 
product theorem $\CE_V \otimes \CF_V$ is projective, and zero in the stable
category. Tensoring with the triangle, we see that both 
$\CE_V$ and $\CF_V$ are idempotent 
modules. It can be seen that the relevant universal properties of the 
triangle also hold. 

The endomorphism ring $\Endul_{kB}(\CF_V)$ is the set of all element of 
$\Endul_{kS}((\CF_V)_{\downarrow S})$ that are stable under the action 
of $T$. The identity element is certainly $T$-stable. 
Let $\CI$ be the maximal ideal in $\Endul_{kS}((\CF_V)_{\downarrow S})$.
Because, by Theorem
\ref{thm:elemain}, $\CI^2 = \{0\}$, any element of $\CI$ that is 
invariant under $T$ is an orbit sum of the $T$-action. These elements form
a maximal ideal of codimension one, and the product of any two elements
in this ideal is zero. 

The idempotent modules for $G$ can be obtained by using the fact that 
$S$ is TI-subgroup (trivial intersection). That is, for any element 
$g \in G$ we have that $S \cap gSg^{-1}$ is $S$ if $g \in N_G(S) = B$ and 
is $\{1\}$ otherwise. Then by the Mackey Theorem, we have that 
\[
((\CF_V)^{\uparrow G})_{\downarrow B} \cong \CF_V \oplus P
\]
where $P$ is projective. Consequently, the induced module 
$(\CF_V)^{\uparrow G})$ has a single nonprojective direct summand
that is $\CF_{V^\prime}$ where $V^\prime = \res_{B,G}^*(V)$.
It follows that $\Endul_{kG}(\CF_{V^\prime})$ is isomorphic to
$\Endul_{kG}(\CF_{V})$.

All of this has a sweeping generalization that is reminiscent of the work
in \cite{Bneuc} and \cite{Cind}.

For notation we say that if $S$ is an elementary 
abelian $p$-subgroup of a group $G$, its diagonalizer $D = D_G(S)$ is
the subgroup of $N_G(S)$ consisting of all elements whose conjugation
action is by a scalar matrix on the $\bF_p$-vector space of $S$ when 
written as an additive group. As in the above example, it is the subgroup
of elements of $N_G(S)$ that acts trivially on $V_S(k)$.  

\begin{thm} \label{thm:genex}
Suppose that $S$ is a normal elementary abelian $p$-subgroup of $G$ and 
that $D = D_G(S)$. Let $U$ be a subvariety of $V_S(k)$ consisting of 
a single point and assume that $U$ is not contained in $\res_{S, R}^*(V_R(k))$
for any subgroup $R$ of $S$. Let $W = \res_{D,S}^*(U)$. 
Let $V = \res_{G, D}^*(W)$. Let $N$ be the kernel of the natural homomorphism 
$\varphi: k_D^{\uparrow G} \to k$ given by $g \otimes 1 \mapsto 1$ for any 
$g \in G$. Then we have a triangle 
\[
\xymatrix{
{} \ar[r] & k_D^{\uparrow G}/N \ar[r]^{\tau} & \CF_W^{\uparrow G}/\tau(N)
\ar[r] & \Omega^{-1}(\CE_W)^{\uparrow G} \ar[r] & {}
}
\]
where $\tau$ is the map induced on quotients by $1 \otimes \tau_W:
k_D^{\uparrow G} \to \CF_W^{\uparrow G}$. In particular, we have that 
\[
\CF_V \cong \CF_W^{\uparrow G}/\tau(N) \ \text{ and } \
\CE_V \cong \CE_W^{\uparrow G},
\]
and the triangle is the triangle of idempotent modules associated to $V$.
\end{thm}

\begin{proof} 
This follows by a very similar argument as in the above example. 
Note that, by an eigenvalue argument,  
$D$ is precisely the subgroup of $G$ that fixes the point
$U$ in $V_S(k)$.  The fact 
that $\CE_V$ is induced from a $kD$-module follows also from Theorem 1.5
of \cite{Bneuc}, which is proved in even greater generality. 
\end{proof}

\begin{rem} \label{rem:exm}
If the group $G$ in the theorem satisfies the Hypothesis \ref{hyp}, then
Theorem \ref{thm:general} assurs us that $\Endul_{kG}(\CF_V)$ has a unique
maximal ideal $\CI$ having codimension one and that $\CI$ is nilpotent. 
Unlike the example we may not assume that $\CI^2 = \{0\}$. For an
example, let $p=2$, and $G = H \times S$ where $H$
is a semidihedral group and $S$ has order $2$. So if 
$V = \res_{G, S}^*(V_S(k))$,
then by Theorem 7.4 of \cite{Ctriv}, $\Endul_{kG}(\CF_V)$ is the nonpositive
Tate cohomology ring of $H$ which by \cite{BC2} has nonzero products in its 
maximal ideal. Note that in this particular case $D_G(S) = G$, so that 
the above theorem says nothing new. 
\end{rem}

\end{document}